\documentclass[a4paper,11pt]{article}
\usepackage[T1]{fontenc}
\usepackage[utf8]{inputenc}
\usepackage{lmodern}
\usepackage[backref=page, colorlinks=true, citecolor=cyan, linkcolor=blue, urlcolor=blue]{hyperref}
\usepackage{amsfonts,amsmath,amssymb,amsthm}
\usepackage{xurl}
\usepackage{soul}
\usepackage{comment}
\usepackage{graphicx}

\newcommand{\skipitems}[1]{%
  \addtocounter{\@enumctr}{#1}%
}

\newtheorem{thm}{Theorem}[section]
\newtheorem{defi}{Definition}[section]

\newtheorem{lem}{Lemma}[section]

\newtheorem{rem}{Remark}[section]

\newenvironment{proofpers}
{\noindent\it\underline{Proof of Theorem \ref{teoremapersistencia}}:\rm}{\hfill$\Box$} 

\newenvironment{proofattr}
{\noindent\it\underline{Proof of Theorem \ref{atractividad}}:\rm}{\hfill$\Box$}

\newenvironment{proofext}
{\noindent\it\underline{Proof of Theorem \ref{extinciongeneral}}:\rm}{\hfill$\Box$} 

\newenvironment{proofperiod}
{\noindent\it\underline{Proof of Theorem \ref{existenciadesolucion}}:\rm}{\hfill$\Box$} 

\newenvironment{proofperext}
{\noindent\it\underline{Proof of Theorem \ref{extincionperiodico}}:\rm}{\hfill$\Box$}

\newcommand{\R}{\mathbb{R}} 
\newcommand{\N}{\mathbb{N}}

\title{A  discrete one-species chemostat model with delayed response in the growth and non-constant supply}
\date{}

\author{Pablo Amster and Mauro Rodriguez Cartabia}

\begin{document}

\maketitle

\begin{center}
Departamento de Matem\'atica, \\
Facultad de Ciencias Exactas y Naturales\\
Universidad de Buenos Aires
and 
IMAS - CONICET\\
Ciudad Universitaria, Pabell\'on I,
(1428) Buenos Aires, Argentina 
\end{center}
\begin{center}
 pamster@dm.uba.ar -- mrodriguezcartabia@dm.uba.ar
\end{center}


 \begin{abstract} 
 A non-autonomous discrete delayed system for a one-species chemostat based on an Ellermeyer model for the continuous case  is studied. 
 Conditions  for the persistence 
 or the extinction of the solutions are obtained 
respectively in terms of the lower and  upper Bohl exponents for a scalar linear equation associated to the problem. 
 Furthermore, the condition for persistence also implies the attractiveness, that is, the existence of a bounded solution that attracts all the others. 
 As a special case, when the nutrient supply is $\omega$-periodic, the picture is complete: the condition for persistence implies the existence of an attractive non-trivial $\omega$-periodic solution, while non-persistence implies extinction.

\medskip

 \textbf{Keywords}: {Discrete one-species chemostat model, Delay-difference equations, Persistence, Extinction, Periodic solutions}.
 \smallskip 
 
\textbf{MSC 2020}: {34K60, 39A23, 92D25}

\medskip 

 \textit{Article published in Communications in Nonlinear Science and Numerical Simulation,  https://doi.org/10.1016/j.cnsns.2025.109487}
 \end{abstract}

\section{Introduction}

The chemostat is an experimental device employed to cultivate microorganisms under controlled environmental conditions. This continuous culturing method is of great relevance in several fields. In biology, it serves as an idealized model for the analysis of microbial ecology, competition, and evolution, as it allows for the growth of cells at a steady state by constantly adding fresh nutrient medium while removing spent culture. From an engineering perspective, it constitutes a simple prototype for continuous bioprocesses, such as those used in industrial fermentation for the production of antibiotics, biofuels, among other biologics. 
In the standard model, it is assumed that  one or more species are cultivated and  that all the nutrients are in abundance, except for a single one, usually called \textit{limiting nutrient} or \textit{substrate}, which is pumped into the culture vessel at a rate $E$. The concentration of nutrient at the feed tank is given by a function $s^0$ and, in order to preserve the volume constant, the liquid containing substrate, biomass and possibly products derived from this biomass, such as insulin, is pumped out at the same rate $E$, often called \textit{washout rate}.  
The standard continuous model 
for a single-species chemostat reads 
 \begin{equation}\label{continuous model}
     \left\{
 \begin{array}{l}
     s'(t)= E(s^0(t)- s(t)) - p(s(t))x(t),\\
     x'(t) = (p(s(t)) - E)x(t),
 \end{array}
 \right. 
 \end{equation}
where the variable $s$ represents the nutrient, $x$  the biomass and the function 
$p:[0,+\infty)\to \R$ regulates the \textit{per-capita} growth and consumption of the 
nutrient. A canonical example of such uptake function is the 
Monod or Michaelis-Menten function, namely
$$p(s):=\frac{p_{\max}s}{k_{s}+s} \quad \hbox{with $p_{\max}>E$ and $k_{s}>0$};
$$
more generally,  usual assumptions in the literature are 
that   $p$ is  bounded,  concave and nondecreasing, with $p(0)=0$. 
From the biological point of view, the last two assumptions are quite clear: 
   $p(0) = 0$ implies that nutrient consumption  cannot occur in the absence of the substrate, while
monotonicity   represents the fact that the uptake rate is directly proportional to the nutrient concentration.
On the other hand, the  boundedness of $p$ reflects the concept of saturation, i.e. there is a finite maximum rate $p_{max}$ preventing the uptake rate from unlimited growth. Finally, 
the concavity assumption is not strictly necessary but  mathematically convenient and also has a biological interpretation: it suggests that the efficiency of the uptake decreases as the concentration increases.    
Also, for convenience it is frequently assumed that $p$ is smooth enough, e.g. of class $C^1$. 

Model (\ref{continuous model}) and several extensions have been widely 
studied in the literature, see e.g. \cite{HRLS} and \cite{Smith-1995}.  
However, works considering the discrete dynamics are more scarce, although such a study
has a clear motivation regarding numerical simulations or discrete-time experiments.  
While classical descriptions of the chemostat are based on continuous 
differential equations, the difference equations models are relevant 
for several reasons: firstly, as mentioned, they are the natural framework for numerical simulation of these continuous systems and, secondly they are directly applicable to scenarios where measurements or experimental interventions occur at discrete time intervals.

A discrete size-structured version of the model was introduced in 
\cite{Gage-1984} and \cite{Smith-1996} for the autonomous case and treated recently in 
\cite{amsterdiscrete}, for the non-autonomous regime. In this latter work, by means of the Perron-Frobenius theorem, 
the problem was reduced to the study of 
the planar system 
\begin{equation}
\label{Planaire-00}
\left\{\begin{array}{rcl}
s_{t+1} &=&Es^0_{t} + (1-E)[s_{t}-x_{t}p(s_{t})],\\ 
x_{t+1} &=& (1-E)\left(1+p(s_{t}) \right)x_{t}, 
\end{array}\right.
\end{equation}
where
 $E\in (0,1)$ 
 and the uptake function $p:[0,+\infty)\to\R$ is of class $C^1$, bounded and satisfies 
 $$p(0)=0, \qquad 0< p'(\xi) \le p'(0),\quad \xi\ge 0. $$
 Also, it is assumed that  
 $s^0$ is bounded between two positive constants, which has a clear interpretation: on the one hand, 
 a positive lower bound ensures that the system is never devoid of the limiting nutrient, which would lead to the trivial outcome of washout and extinction of the microbial species. 
On the other hand, an upper bound reflects the finite solubility of the nutrient in the medium and 
the practical physical constraints of the experimental device.

A central role for the analysis of (\ref{Planaire-00})  is played by the function 
\begin{equation}\label{construccionz} 
z_{t+1}:=E\sum_{k=-\infty}^{t}(1-E)^{t-k}s^0_k,
\end{equation}
which is readily proven to be the unique bounded solution of the equation
\begin{equation}\label{eq: 4.4}
    z_{t+1}=(1-E)z_{t}+Es_{t}^0,\quad t\in \mathbb Z,
\end{equation}
and corresponds to the amount of nutrient in absence of biomass, leading to the so-called \textit{washout solution} $(z,0)$ of the system. 
It is verified that $z$ attracts exponentially all the other solutions of \eqref{eq: 4.4} as $t\to+\infty$ and, moreover,  that the condition
$$p'(0) z_{\sup}\le 1$$
with $z_{\sup}:=\displaystyle{\sup_{t\ge 0}z(t)}$ ensures that the solutions of (\ref{Planaire-00}) remain non-negative, provided 
that the initial conditions $(s_0,x_0)$ are non-negative and such that $s_0+x_0\le z_0$. 
Since $p'(\xi)\le p'(0)$ (which, in turn, is trivially satisfied if $p$ is concave), the preceding assumption 
simply says that
\begin{equation}\label{pz}
0<    p'(\xi) z_t\le 1,\qquad {\xi, t\ge 0}.
\end{equation}
This inequality, along with the natural condition $p(0)=0$, will constitute the only standing hypotheses on the $C^1$ map $p$ throughout this work.

In the previous context, sufficient conditions for the persistence or the extinction of the biomass can be expressed respectively in terms of the discrete  lower and upper Bohl exponents associated 
to the linear equation 
$$c_{t+1}= (1-E)(1+p(z_t))c_t.$$ 
According to the literature
\cite{Babiarz:15}, 
such exponents can be
defined by
$$
\underline{\beta}(p(z)):= 
\liminf_{t_1, t_2-t_1\to +\infty} 
\left( \prod_{k=t_1+1}^{t_2} [
(1-E)(1 + p(z_k))] \right)^{\frac 1{t_2-t_1}},$$
$$
\overline{\beta}(p(z)):= 
\limsup_{t_1, t_2-t_1\to +\infty} 
\left( \prod_{k=t_1+1}^{t_2} [
(1-E)(1 + p(z_k))] \right)^{\frac 1{t_2-t_1}}.$$
For the reader's convenience, 
let us recall  the meaning of the Bohl exponents associated to
a linear scalar nonautonomous equation. 
In the ODEs setting, consider the homogeneous equation 
\begin{equation}\label{homog}
    x'(t)=a(t)x(t)
\end{equation}
for 
some continuous function $a:[0,+\infty)\to \R$ and define 
$$\underline b:= \liminf_{s, t-s\to+\infty} \frac 1{t-s}\int_s^t a(\xi)\, d\xi, 
$$
$$\overline b:= \limsup_{s, t-s\to+\infty} \frac 1{t-s}\int_s^t a(\xi)\, d\xi.
$$
It is straightforward to see that if $\underline b> 0$ or $\overline b< 0$ then the solutions of (\ref{homog}) have an exponential behavior as $t\to+\infty$, respectively diverging or converging to $0$. Furthermore, for any continuous bounded function $c:[0,+\infty)\to \R$  the non-homogeneous equation 
$$x'(t)=a(t)x(t)+c(t)$$
has a unique solution defined for $t\ge 0$, which is bounded  in the first case, while all the solutions are bounded, in the second case. 
For mathematical convenience, when dealing with linear difference equations it is usual to consider, instead, the exponential of the previous quantities, which leads to the previous definitions of $\underline\beta$ and $\overline \beta$. 
 
As mentioned, the Bohl exponents proved to be a useful tool to describe the dynamics of the chemostat. Indeed, from the results in \cite{amsterdiscrete} it is directly deduced  that the  conditions 
$$\underline \beta >1,\qquad \overline \beta <1$$
imply, respectively, that the $x$-coordinate of the positive trajectories remains bounded away from $0$ or tends to $0$ as $t\to+\infty$. 
Thus, the Bohl exponents play a central role in the study of the dynamics of the system. 
A special situation occurs when $s^0$ is 
$\omega$-periodic for some $\omega\in\mathbb N$, since in this case it is seen that $z$ is also $\omega$-periodic and $\underline \beta=\overline \beta:=\beta$. The value $\beta=1$ is a threshold and {in \cite{amsterdiscrete}} it is proven that  the persistence assumption $\beta>1$ yields the existence of a unique (attractive) positive $\omega$-periodic solution, while the assumption $\beta < 1$ implies that all the trajectories converge to the washout solution. However, nothing is said in \cite{amsterdiscrete} 
about the limit case $\beta=1$.  

A delayed version of the previous discrete model was firstly introduced in \cite{AR24}, based on a model by Caperon studied in \cite{amster2020dynamics}. 
However, unlike the continuous $\omega$-periodic model, the assumption $\beta>1$ does not seem to suffice 
to ensure  the existence of a nontrivial $\omega$-periodic solution. 
Here, we shall extend  the above described results to a different  discrete delayed system, in this case based on a continuous model proposed by Ellermeyer in \cite{ellermeyer1994competition}, which was also studied in \cite{AR20,cartabia2023persistence,cartabia2025uniform}. Namely, we shall consider the problem
\begin{equation}\label{modeloprincipal-intro}
  \left\{\begin{aligned}
  s_{t+1}&=Es_{t}^0+(1-E)(s_{t}-x_{t}p(s_{t}))\\
  x_{t+1}&=(1-E)x_{t}+x_{t-r}p(s_{t-r})(1-E)^{r+1}
  \end{aligned}\right.
\end{equation}
where $r\in \mathbb N_0$   is a fixed delay. 
Appropriate Bohl exponents will be set in terms of an auxiliary function $\varphi$, which shall be defined in order to take the delay into account. 
The definition of $\varphi$ follows the ideas introduced in \cite{cartabia2023persistence} for 
the continuous model, conveniently adapted to the discrete setting. 
 As we shall see, 
 in the non-delayed case it is verified that $\varphi\equiv 1$ and the results 
 established below slightly improve those obtained in \cite{amsterdiscrete}. 
 Thus, the methods  in the present work can be regarded as a 
combination of those employed in 
\cite{cartabia2023persistence}
and in \cite{amsterdiscrete}.

In more precise terms,  the main contributions of the present work can be described as follows. 
A discrete version of the Ellermeyer delayed chemostat model presented in \cite{ellermeyer1994competition} is  studied.
Analogous versions of the known results obtained in \cite{cartabia2023persistence, cartabia2025uniform} for the continuous case are obtained, bridging gaps between theoretical and computational 
approaches. 
On the one hand,  it will be shown by
Theorem \ref{teoremapersistencia} 
that the condition $\underline \beta >1$ 
for persistence is necessary and even more, the persistence of a single solution implies the uniform persistence of the entire system.   
Interestingly, this 
fact seems to have not been noticed previously
 in the literature, not even for the continuous undelayed model
(\ref{continuous model}) treated for example in  \cite{ellermeyer2001persistence}, 
 although the proof is very  simple in this particular case (see Remark \ref{continuous} below). 
 {Further, Theorem \ref{atractividad} proves that the persistence of system (\ref{modeloprincipal-intro}) implies that any positive solution attracts geometrically all the other positive 
 solutions. 
 A sort of opposite situation is verified in Theorem \ref{extinciongeneral}, which 
 proves that the inequality $\overline \beta<1$ provides a sufficient condition for the extinction of all positive solutions. 
 }


{Despite the previously mentioned results, it will be  shown with an example that the case $\overline \beta\ge 1\ge \underline \beta$ 
may lead to  situations in which some trajectories persist while some others become extinct. 
  However, this cannot occur in the biologically 
  relevant periodic case, namely  
when the input $s^0$ in (\ref{modeloprincipal-intro}) is an $\omega$-periodic function. Here, 
the upper and lower Bohl exponents coincide and
the dynamics can be fully characterized. As shown below in 
Theorem \ref{persitenciaperiodico},  in this case persistence implies the existence of an (attractive) $\omega$-periodic positive solution and, on the other hand, 
Theorem \ref{extincionperiodico}  states that 
 non-persistence   always implies extinction. 
 Furthermore, the conditions for the persistence/extinction scenarios can be expressed in terms of the geometric mean of the function $(1-E) (1+\varphi p(z))$.
 Together  with the result for the general case, the existence of a nontrivial periodic solution implies its attractiveness and the uniform persistence of 
the system. This is consistent with the known results for the continuous 
 undelayed case, see  e.g. \cite[Cor 2.3]{wolk} but, to our knowledge, constitutes 
 an improvement when a positive delay is included to the model. 
}

The {rest of the} 
paper is organized as follows.
The next section is devoted to describe  a discrete delayed model for a one-species chemostat based on the continuous model introduced in \cite{ellermeyer1994competition}. 
In section \ref{prelim}, we give the basic definitions and notations that are needed to establish our 
main results, which shall be presented in section \ref{main}.
The periodic case is addressed in  
a separate subsection.  For the sake of clarity,  section 
\ref{technical} is devoted to establish 
the technical lemmata that  shall be used in
section \ref{proofs}, where  the main results are proven.  
Finally, numerical simulations to ilustrate the theoretical results are presented in section \ref{simulaciones}.

\section{A delayed discrete system based on Ellermeyer's continuous model 
}
\label{sec-2}

   In this section, we shall deduce a delayed discrete model for a one-species chemostat based on  \cite{ellermeyer1994competition}. 
With this in mind, we shall assume a number  
$N+1$  of population classes, a washout   constant $\hat E\in (0,1)$, and growth function $\hat p$. The $0$-th class means maturity and the classes $1,\ldots, N$ mean immaturity, where the first one is the most immature and the $N$-th one is the closest to being mature. The biomass of of class $k$ shall be denoted by $x^{(k)}_t$. Each step between classes has length $h$, and the 
discrete model approaches to the continuous one by letting $h\to 0$, with  
$$Nh=\textit{constant}:=\hat r.$$ 
Thus, the constant $\hat r$ represents the total time required by the species to reach maturity.
We remark that if $N=0$ there is a unique maturity class and no delay. 
Here, we shall assume that 
each population class evolves completely at each time step, except the maturity one. Moreover, 
the $0$-th class    consumes substrate and generates an amount $h\hat p(s_t)x^{(0)}_t$ of new biomass. Thus we have:
\begin{equation*}
\left\{\begin{aligned}
  x^{(1)}_{t+h}&\approx   x^{(0)}_t\,h\hat p(s_t)(1-h\hat E),\\
  x^{(2)}_{t+2h}&\approx x^{(1)}_{t+h}(1-h\hat E),\\
  \vdots\\
  x^{(N)}_{t+Nh}&\approx x^{(N-1)}_{t+(N-1)h}(1-h\hat E),\\
  x^{(0)}_{t+(N+1)h}&\approx \left(x^{(0)}_{t+Nh}+x^{(N )}_{t+Nh}\right)(1-h\hat E),
\end{aligned}
\right.
\end{equation*}
and inductively 
{
\begin{align*}
    x^{(0)}_{t+ h}&\approx  (1-h\hat E)x^{(0)}_{t }+x^{(N )}_{t } (1-h\hat E)\\
    &\approx (1-h\hat E)x^{(0)}_{t }+x^{(N-1)}_{t-h} (1-h\hat E)^2\\ 
    &\approx \ldots \\
    &\approx (1-h\hat E)x^{(0)}_{t}+ x^{(0)}_{t -\hat r} h\hat p(s_{t-\hat r}) (1-h\hat E)^{N+1}.
\end{align*}
}
Thus, 
$$x^{(0)}_{t+h}-x^{(0)}_{t }\approx -h\hat Ex^{(0)}_{t }+ x^{(0)}_{t-\hat r}h\hat p(s_{t-\hat r})(1- h\hat E)^{N+1}$$
which, in turn, implies 
$$\frac{x^{(0)}_{t+h}-x^{(0)}_{t }}{h}\approx -\hat Ex^{(0)}_{t}+x^{(0)}_{t-\hat r}\hat p(s_{t-\hat r})(1-h \hat E)^{\hat r/h+1}.$$ 
By letting  $h  \to 0$, the following Ellermeyer equation is obtained \cite{ellermeyer1994competition}:
$$\frac{d}{dt}x _{t}=-\hat Ex _{t}+x _{t-r}\hat p(s_{t-\hat r})e^{-\hat r \hat E}.$$

Moreover, according to 
\cite[Eq- (1.1)]{ellermeyer1994competition}, 
it is deduced that  
\begin{align*}
  \sum_{k=1}^{N} x^{(k)}_t &\approx \sum_{k=1}^{N} x^{(1)}_{t-(k-1)h}(1-h\hat E)^{k-1}  \approx \sum_{k=1}^{N} x^{(0)}_{t-kh}h\hat p(s_{t-kh})(1-h\hat E)^{k}  
\end{align*}
and, again, letting $h\to 0$  a formula for the concentration of nutrient stored internally by the species reads
$$y(t)=\int_0^{\hat r} x_{t-\sigma}\hat p(s_{t-\sigma})e^{ -\sigma \hat E}\, d\sigma$$
and has a fundamental role in the analysis of the continuous model.

{On the other hand, 
the substrate consumption by the $0$-th class
yields the following formula:  
$$s_{t+h}\approx h\hat Es_t^{0}+(1-h\hat E)s_t-x_t^{(0)}h\hat p(s_t)(1-h\hat E),$$
that is
\begin{align*}
  \frac{s_{t+h}-s_t}{h}&\approx \hat E(s_t^{0}-s_t)-x_t^{(0)}\hat p(s_t)(1-h\hat E) 
\end{align*} 
and we recover the equation for $s$,
$$\frac{d}{dt}s_t=E(s_t^{0}-s_t)-x_t\hat p(s_t)$$
as $h\to 0$.}

Since we shall be dealing with a discrete version of the model, let us fix a positive constant $h$, which is assumed to be small. Then, by defining $E:=h\hat E$, $p:=h\hat p$, and $r:=\hat r/h=N$, we get the system:
\begin{equation*} 
  \left\{\begin{aligned}
  s_{t+h}&=  Es_{t}^0+(1-  E)(s_{t}- x_{t}p(s_{t}))\\
  x_{t+h}&=(1-  E)x_{t}+ x_{t-h r} p(s_{t- hr})(1-  E)^{r+1}.
  \end{aligned}\right.
\end{equation*}

Note that the value of $h$ only appears in the time domain and does not affect the dynamics of the model. Therefore, for simplicity, we can fix $h=1$ in order to obtain the above mentioned system (\ref{modeloprincipal-intro}) with a delay $r$:
\begin{equation}\label{modeloprincipal}
  \left\{\begin{aligned}
  s_{t+1}&=Es_{t}^0+(1-E)(s_{t}-x_{t}p(s_{t}))\\
  x_{t+1}&=(1-E)x_{t}+x_{t-r}p(s_{t-r})(1-E)^{r+1}.
  \end{aligned}\right.
\end{equation}
 As usual, an initial condition for (\ref{modeloprincipal}) shall be given by a vector 
  $(s^{in},x^{in})\in [0,+\infty)^{r+1}\times [0,+\infty)^{r+1}$, namely
  $$(s^{in},x^{in})= (s^{in}_{-r},\ldots, s^{in}_{0},x^{in}_{-r},\ldots, x^{in}_{0}).
  $$
  As before, it is noticed that, 
  unlike the continuous model, extra assumptions are needed in order to   guarantee that the substrate remains nonnegative for all nonnegative values of $t$ (see Lemma \ref{posit} below). This need is quite obvious even for linear difference equations: for 
  instance, when dealing with the problem $u_{t+1} = (1 + a_t)u_t$, in order to preserve positiveness it is required that $a_t>-1$ for all $t$.

\section{Preliminaries}
\label{prelim}

Here, we introduce the basic definitions and notations that shall be used throughout  the paper. In the first place, analogously to the continuous case, the quantity
  \begin{equation}\label{definicionY}
    y_{t+1}:=\sum_{k=0}^{r-1} x_{t-k } p(s_{t-k })(1-E)^{k+1}
\end{equation}
represents the nutrient that is internally stored by the species and 
shall be crucial in our analysis. Moreover, due to the presence of delay, the linear equation 
that will allow us to compute an accurate Bohl  exponent now reads
$$c_{t+1}=(1-E)c_{t}+c_{t-r}p(z_{t-r})(1-E)^{r+1}.
$$
For convenience, let us define 
\begin{equation}\label{definicionPhi}
    \varphi_{t-r}:=\frac{c_{t-r}}{c_{t}}(1-E)^{r} \qquad t\ge 0,
\end{equation} 
so the previous equation can be written as 
\begin{equation}\label{pseudo-linear}
c_{t+1}=(1-E)(1+\varphi_{t-r} p(z_{t-r}))c_{t}.    
\end{equation}
Thus, if the initial condition for $c$ is chosen in such a way that $c_0=1$, then
it follows that
$$c_{t+1}=(1-E)^{t+1}\prod_{k=-r}^{t-r}( 1+\varphi_{k} p(z_{k})).$$
This yields a useful fixed-point
identity for   $\varphi$: namely, writing  
$$
  \varphi_{t+1} =\frac{(1-E)^{t+1}\prod_{k= -r}^{t-r}( 1 +\varphi_{k} p(z_{k}))}{(1-E)^{t+1+r}\prod_{k=-r}^{t}( 1 +\varphi_{k} p(z_{k}))}(1-E)^{r },
$$
it is deduced that 
$$\varphi_{t+1}= \prod_{k=t+1-r}^{t}( 1+\varphi_{k} p(z_{k}))^{-1}.$$

 A similar expression can be obtained for 
$x$, by letting 
\begin{equation}\label{definicionPsi}
    \psi_t:=\frac{x_t}{x_{t +r}}(1-E)^{r},
\end{equation}  
so we may write
$$x_{t+1}=(1-E)x_{t}+x_{t-r}p(s_{t-r})(1-E)^{r+1}=(1-E)x_{t}( 1 +\psi_{t-r} p(s_{t-r}))$$
and consequently
 \begin{equation}\label{expresionxt}
  \begin{aligned}
  x_{t+1} =x_{ 0}(1-E)^{t+1}\prod_{k=-r}^{t-r}( 1 +\psi_{k } p(s_{k }))
  \end{aligned}
\end{equation}
whence, as before,   
\begin{align*}
  \psi_{t+1}& = \prod_{k=t+1-r}^{t}( 1+\psi_{k} p(s_{k}))^{-1} .
\end{align*}
For the reader's convenience, 
it is worth noticing  that, as usual, for an arbitrary set of coefficients $\{j_k\}_{k\in K}$
it is understood 
that, if $K=\emptyset$, then 
   $$\prod_{k\in K}j_k=1.$$
In particular, for the undelayed case $r=0$ it follows that   $\varphi_t=\psi_t\equiv 1$, Obviously, this can be also deduced directly from \eqref{definicionPhi} and \eqref{definicionPsi}. 
\begin{rem}
    The roles of the mappings $\varphi$ and $\psi$ can be understood as follows. Let us firstly observe, in the continuous undelayed model (\ref{continuous model}), 
    that adding the two equations of the system yields the following equation for $v:=s+x$:
    $$v'(t)= E(s^0(t)-v(t)).$$
    This represents a key aspect of the model and allows to ensure that the total amount of the substrate and biomass approaches, as $t\to+\infty$, to the function  $z$, namely the first coordinate of the washout solution. The Bohl exponents are computed from the second equation of the system, in which the $s$ variable is replaced by $z$.
    When a delay is present, the new variable $y$ detailed  in section \ref{sec-2} is introduced 
    in order to fill the gap between $s+x$ and $z$ and, in some sense, the 
    functions $\varphi$ and $\psi$ introduce a ``correction" into the 
    above mentioned equation due to the delay.     
    This machinery was already developed in the continuous case in \cite{ellermeyer1994competition, cartabia2023persistence}
    whose procedure is emulated here for the discrete problem. 
    
\end{rem}
As it is usual in the literature on chemostat models (see e.g. \cite{AR24}), the notions of persistence or extinction of the solutions always refer to the biomass, that is:

\begin{defi}
    \begin{enumerate}
        \item A solution $(s_t,x_t)$ of (\ref{modeloprincipal}) shall be called (strongly) persistent if 
        $\displaystyle\liminf_{t\to\infty} x(t)>0$. \item System (\ref{modeloprincipal}) shall be called persistent if all its nontrivial solutions are persistent. 
        \item  System (\ref{modeloprincipal}) shall be called uniformly persistent if there exists $\delta>0$ such that each nontrivial trajectory $(s_t,x_t)$ satisfies  $\displaystyle\liminf_{t\to\infty} x(t)\ge \delta$.
    \end{enumerate}
\end{defi}
 \begin{defi}
     \begin{enumerate}
        \item A nontrivial solution $(s_t,x_t)$ of (\ref{modeloprincipal}) goes to extinction if  $\displaystyle \lim_{t\to\infty} x(t) =0$. 
        \item System (\ref{modeloprincipal}) goes to extinction if all its nontrivial solutions go to extinction. 
     \end{enumerate}         
 \end{defi}

To conclude this section, 
let us recall that, as mentioned in the introduction, 
it shall be assumed 
throughout the paper that 
$s^0_t$ is bounded between two positive constants and that $E\in (0,1)$.

\section{Main results}

\label{main}

In this section, we shall establish our main 
results. In a first subsection it shall be stated  that the persistence of at least one solution implies the uniform persistence of the system and, furthermore, 
that any nontrivial trajectory attracts all the others. 
This can be expressed in terms of some appropriate lower Bohl exponent associated to the problem, while  a reversed inequality for the upper exponent will imply that the system is driven to extinction. 
Furthermore, it will be shown with a simple example 
that it may happen that trajectories do not persist in a non-extinction scenario; however, in the periodic  case the picture is complete: the main result in the second subsection establishes that   non-persistence always implies extinction.

\subsection{The general non-autonomous case}

\begin{thm}\label{teoremapersistencia}
The following statements for model
\eqref{modeloprincipal} are equivalent:   
\begin{enumerate}
  \item\label{z} There {exists a solution  $(s_t,x_t)$ such that
  $$\liminf_{t\to\infty}x_t>0;$$}
  \item\label{a} The system is   persistent;
  \item\label{b} The system is  uniformly persistent; 
  \item\label{c} There exists $\delta>0$ 
  such that, 
for every  $R,  \alpha>0$, there exists $T=T(R,\alpha)\geq 0$ such that if a trajectory $(s_t,x_t)$ verifies $\|(s^{in},x^{in})\|\leq R$ and $x_0\geq \alpha$ then $x_t\geq \delta $ for all $t\geq T$;
  \item\label{d}  There exist constants  $\eta>0$ and $T>r$ such that   
\begin{equation}\label{condicionpersistencia}
  \prod_{k=t_1+1}^{t_2}(1 +\varphi_{k-r}p(z_{k-r}) )>\left(\frac{1+\eta}{1-E}\right)^{ t_2-t_1  }
\end{equation}
for all $t_1>T$ and $t_2-t_1>T$. 
\end{enumerate}
\end{thm}

\begin{rem}
Omitting the fact that $\varphi$ depends on $c$, equation
\eqref{pseudo-linear}
may be regarded as linear, so the associated  
lower and Bohl exponents can be computed as
$$\underline{\beta}(\varphi p(z)):= \liminf_{t_1, t_2-t_1\to +\infty} 
\left(\prod_{k=t_1+1}^{t_2}[(1-E)(1 +\varphi_{k}p(z_{k}))] \right)^{\frac 1{t_2-t_1}}.$$
and
$$\overline{\beta}(\varphi p(z)):= \limsup_{t_1, t_2-t_1\to +\infty} 
\left(\prod_{k=t_1+1}^{t_2}[(1-E)(1 +\varphi_{k}p(z_{k}))] \right)^{\frac 1{t_2-t_1}}.$$
Thus, it is immediately seen that condition \ref{d} in 
the previous theorem can be  written as 
$\underline \beta>1$. 
In particular,  when $r=0$, it was already mentioned that $\varphi\equiv 1$, so the result is consistent with the corresponding one in \cite{amsterdiscrete}.

\end{rem}


\begin{rem}
\label{continuous} 
Consider the continuous model \eqref{continuous model}, with   $E>0$ and $s^0$  continuous and bounded between two positive constants, then the unique bounded solution of the linear equation $w'(t)=E(s^0(t)-w(t))$ is now given by 
 $$z(t):=\int_{-\infty}^t s^0(\xi) e^{E(\xi-t)}\, d\xi,$$
which is exponentially attractive as $t\to+\infty$. In particular, if $(s,x)$
is a nonnegative solution of the system, then $$(z(t) - s(t)- x(t) )= (z(0)-s(0)-x(0))e^{-Et}.$$  
Thus, if $x(t)\ge c >0$ for $t\gg 0$, 
then for some $\eta >0$ it is verified that $p(z(t)) > p(s(t)) +\eta$ for large values of $t$. Integrating the equation for the biomass it follows, for arbitrary $t_2>t_1\gg 0$ that
$$\ln 
\left(\frac{x(t_2)}{x(t_1)} 
\right) = \int_{t_1}^{t_2} (p(s(\xi))-E)\, d\xi.
$$
Taking now into account that $x$ is also bounded from above, the previous identity yields
$$\liminf_{t_1, t_2-t_1\to+\infty}
\frac 1{t_2-t_1}\int_{t_1}^{t_2}  {(p(z(\xi))-E)}\, d\xi \ge  \eta.
$$
As mentioned in the introduction, for the continuous case, the 
left-hand side of the previous inequality is the lower Bohl exponent associated to the continuous linear equation $c'(t)=(p(z(t))-E)c(t)$. In other words,  the existence of a persistent trajectory implies 
that such exponent is positive and, in turn, this implies the uniform persistence of the system.

\end{rem}

The following result shows that, furthermore, the condition for persistence implies 
the attractiveness. In more precise terms, we have: 

\begin{thm}\label{atractividad}
    In the context of the  previous theorem, 
     the equivalent conditions 1--5  imply that any positive solution attracts geometrically all the positive trajectories. 
\end{thm}

To conclude this subsection, 
we shall establish a result that, roughly speaking, complements 
Theorem \ref{teoremapersistencia}: again, analogously to the nondelayed case, it will be shown that the condition 
$\overline \beta(\varphi p(z)) < 1$ implies the extinction of all the solutions.

\begin{thm}\label{extinciongeneral}
    Assume there
 exist constants  $\eta>0$ and $T>r$ such that   
\begin{equation}\label{condicion-ext}
  \prod_{k=t_1+1}^{t_2}(1 +\varphi_{k-r}p(z_{k-r}) )<\left(\frac{1-\eta}{1-E}\right)^{ t_2-t_1  }
\end{equation}
for all $t_1>T$ and $t_2-t_1>T$. Then all the nonnegative solutions $(S,x)$ of (\ref{modeloprincipal}) satisfy
$$\lim_{t\to\infty} x_t=0.
$$

\end{thm}

\subsection{{The neither-nor scenario: a case of neither persistence nor extinction}}

With the results of the preceding subsection in mind, let us  show with an example in which
trajectories are neither persistent 
nor they become extinct. To this end, let  $p(s)=s$, fix 
nonzero initial conditions  $(s^{in},x^{in})$ and consider the function
$$s_t^{0}=\left\{\begin{array}{cl}
    E^{-1}(1-E)^{-2r-2} &\text{ if }2^{2n}\leq t<2^{2n+1},\,n\in\N,  \\
    E/2 &\text{otherwise.} 
\end{array}\right.$$
Next, fix $T>0$ and consider $n\in\N$ large enough such that $T<2^{2n}$ and 
$$(1-E)^{2^{2n}}s_1^{0}<E/2.$$
We remark that  $s_1^{0}$ is an upper bound for the function $z_t$ defined by formula \eqref{construccionz} and that $s^{0}_t=E/2$ if $2^{2n+1}\leq t<2^{2(n+1)}.$  Therefore, for $3\cdot 2^{2n}\leq k<2^{2(n+1)}$, it is deduced that 
\begin{align*}
    z_k&=E\sum_{j=-\infty}^{k-1}(1-E)^{k-1-j}s^{0}_j\\
    &=E\sum_{j=2^{2n+1}}^{k-1}(1-E)^{k-1-j}s^{0}_j+E\sum_{j=-\infty}^{2^{2n+1}-1}(1-E)^{k-1-j}s^{0}_j\\
    &\leq E\frac{1}{E} E/2+E\frac{(1-E)^{2^{2n}}}{E}s^{0}_1\\
    &<E.
\end{align*}
Furthermore, 
\begin{align*}
    \prod_{k=3\cdot 2^{2n}+1+r}^{2^{2(n+1)}+r}(1 +\varphi_{k-r} z_{k-r}  )&< \prod_{k=3\cdot 2^{2n}+1}^{2^{2(n+1)}}(1 + z_{k}  ) <\left(1+E\right)^{ 2^{2n}  } <\left(\frac{1+\eta}{1-E}\right)^{ 2^{2n}  }
\end{align*}
for any $\eta>0$ and, by Theorem \ref{teoremapersistencia}, the system is not persistent.

It remains to show that the solution does not become extinct. Indeed, fix  $n\in\N$ and observe that 
$$s_{t}\geq Es_{t-1}^{0}\geq (1-E)^{-2r-2} $$
if  $2^{2n}< t\leq 2^{2n+1}$. As before,  
$$x_{t}\geq (1-E)^{t}x_0 $$
for all $t\geq 0$. Therefore, set $m=m(n)$ the integer such that  $m(r+1)< 2^{2n}\leq (m+1)(r+1)$. This implies that $s_{2^{2n+1}-m(r+1)}\geq (1-E)^{-2r-2} $ and, finally, 
\begin{align*}
    x_{2^{2n+1}}&=(1-E)x_{2^{2n+1}-1}+x_{2^{2n+1}-(r+1)}s_{2^{2n+1}-(r+1)}(1-E)^{r+1}\\
    &\geq x_{2^{2n+1}-(r+1)}s_{2^{2n+1}-(r+1)}(1-E)^{r+1}\\
    &\geq  x_{2^{2n+1}-2(r+1)}s_{2^{2n+1}-2(r+1)}(1-E)^{r+1}s_{2^{2n+1}-(r+1)}(1-E)^{r+1}\\
    & =x_{2^{2n+1}-m(r+1)}\prod_{k=1}^m s_{2^{2n+1}-k(r+1)}(1-E)^{r+1}\\
    & >(1-E)^{2^{2n+1} -m(r+1)}x_0 \prod_{k=1}^m  (1-E)^{-2r-2}(1-E)^{r+1}\\
    &= (1-E)^{2^{2n+1}-2m(r+1)}x_0  \\
    &\geq (1-E)^{2(r+1)}x_0
\end{align*}
by definition of $m$. Since this bound does not depend on $n$, we conclude that  
$$\limsup_{t\to +\infty}x_t>0.$$

\subsection{Periodic case}

Here, it shall be assumed that $s^0$ is $\omega$-periodic for some $\omega\in \mathbb N$. 
It is readily verified that $z$ is 
$\omega$-periodic and, imitating the procedure in \cite{cartabia2023persistence}, we may choose $c$ in  such a way that $\varphi$ is also $\omega$-periodic. 
Thus,  the values $\overline \beta$ 
and $\underline \beta$ previously defined are equal and   
 coincide with the geometric mean of the function $(1-E) (1+\varphi p(z))$, namely
$$\langle (1-E) (1+\varphi p(z))
\rangle:= \left(\prod_{k=0}^{\omega-1}[(1-E)(1 +\varphi_{k}p(z_k))] \right)^{\frac 1{\omega}}.
$$

\begin{thm}\label{persitenciaperiodico}
In the previous context, the following statements are equivalent: 

\begin{enumerate}\label{existenciadesolucion}
    \item System \eqref{modeloprincipal} admits at least one positive $\omega$-periodic  solution. 
    \item System \eqref{modeloprincipal} has a unique positive $\omega$-periodic  solution, which attracts geometrically all the nontrivial trajectories. 
\item System \eqref{modeloprincipal} is persistent. 
\item System \eqref{modeloprincipal} is uniformly persistent. 
\item At least one nontrivial solution persists. 
\item $\langle (1-E) (1+\varphi p(z))
\rangle>1.$
    
\end{enumerate}
    
\end{thm}

\begin{thm}\label{extincionperiodico}
    If $\langle (1-E) (1+\varphi p(z))
\rangle\le 1$, then every trajectory of \eqref{modeloprincipal}
goes to extinction. \end{thm}

\section{Technical results}
\label{technical}

 This section is devoted to establish a series of technical results that shall be employed for the proofs of the main theorems. 
 
\begin{lem}\label{lemaclave}
Let $(s_t,x_t)$ be a solution of system \eqref{modeloprincipal}, let $y_t$ be given by \eqref{definicionY} and let $z_t$ be given by \eqref{construccionz}. Then
$$s_{t+1}+x_{t+1}+y_{t+1}-z_{t+1}=(1-E)^{t+1}(s_0+x_0+y_0-z_0).$$ 
{In particular, if 
$s_0+x_0+y_0\le z_0$, then $s_t+x_t+y_t\le z_t$ for all $t\ge 0$.} 
\end{lem}

\begin{proof}
Observe that   
\begin{align*}
  s_{t+1}+x_{t+1}+y_{t+1}&= Es_{t }^0+(1-E)(s_{t }-x_{t }p(s_{t }))\\
  &\quad +(1-E)x_{t-1}+x_{t-r}p(s_{t-r})(1-E)^{r+1}\\
  &\quad +\sum_{k=0}^{r-1} x_{t-k } p(s_{t-k })(1-E)^{k+1}\\
  &= Es_{t}^0+(1-E)(s_{t}+x_{t})+\sum_{k=1}^{r } x_{t-k } p(s_{t-k })(1-E)^{k+1}\\
  &= Es_{t}^0+(1-E)(s_{t}+x_{t})+y_{t}(1-E) \\
  &=Es^0_{t}+(1-E)(s_{t}+x_{t}+y_{t}).
\end{align*}
This proves that the function 
$s_{t}+x_{t}+y_{t}$ is a solution of (\ref{eq: 4.4}), which yields the identity 
$$z_{t+1} - (s_{t+1}+x_{t+1}+y_{t+1}) =(1-E)^{t+1}
[z_{0} - (s_{0}+x_{0}+y_{0})]
$$
and the conclusion follows.
\end{proof}

\begin{lem}
\label{posit}
Assume that $z_t p'(\xi) \le 1$ for all $t$ and all $\xi$. 
Then for any nonnegative initial conditions  with $s_0+x_0 + y_0 \le z_0$, the trajectory satisfies $x_t\ge 0$ and $s_t>0$ for all $t> 0$. 
\end{lem}
\begin{proof}
    By induction, assume that the conclusion holds for all 
    $j\le t$ and observe, by the previous lemma, that 
    $s_t+x_t + y_t \le z_t$ which, in particular, implies $x_t\le z_t$.  
    It is clear that $x_{t+1}\ge 0$ and, moreover, for some mean value $\xi$ it is seen that
    $$s_{t+1} = 
    Es_t^0 + (1-E) (s_t - x_tp'(\xi)s_t)
$$
$$> (1-E) (1 - z_tp'(\xi))s_t\ge 0.$$
\end{proof}
The following result is based on  \cite[Lemma 4.3]{cartabia2023persistence}. 
\begin{lem}\label{lemaida}
Consider the functions defined by recursion as
\begin{align*}
  \varphi_{t+1}&:=   \prod_{k=t+1-r}^{t}(1 +\varphi_kf_k)^{-1},\\
  \psi_{t+1}&:=  \prod_{k=t+1-r}^{t}(1 +\psi_kg_k)^{-1} 
\end{align*}
with $0\le f_t \le M$  for all $t\geq 0$ and  $f_t,\,g_t,\,\varphi_t,\,\psi_t $ nonnegative for all $t\geq 1-r$. Furthermore, assume  there exists a constant  $\varepsilon>0$ such that  $f_t>g_t+\varepsilon$ for all $t\geq 0$. 
Then there exist constants $\eta>0$ and $T\geq  r$ such that  
$$ \prod_{k=t_1+1}^{t_2}\frac{1 +\varphi_{k-r}f_{k-r}}{1 +\psi_{k-r}g_{k-r}}>(1+\eta)^{2(t_2-t_1) }$$
for all  $t_1\geq T$ and $t_2-t_1\geq T$.
\end{lem}

\begin{proof}
Define the positive constants  
\begin{align*}
  m&:=(1+M) ^{-r+1},\\
  \kappa&:=1+\frac{\varepsilon}{2M}
\end{align*}  
  and $\eta, T$ as the unique solutions of the respective equations
  $$
 \quad (1+\eta)^3= \min\left\{\left(\frac{\kappa/m+M+\varepsilon}{1/m+M}\right)^{1/r},1+\frac{m \varepsilon/2}{1+M}\right\},$$
 $$
   (1+\eta)^{T-3(r-1)}=(1+M)^{r-1}.$$
Further, observe  that   $\varphi_t\leq  1$ implies $\varphi_t\geq  m$.
We split the proof into three steps. 
\medskip 

\textit{Step 1.} Let $t\geq r$, then  there are two cases to study. If $\kappa\varphi_t\leq \psi_t$, then
\begin{align*}
  \prod_{k=t+1-r}^{t}\frac{1+\varphi_kf_k}{1+\psi_kg_k}&=\left(\frac{1+\varphi_tf_t}{1+\psi_tg_t}\right)\prod_{k=t+1-r}^{t-1}\frac{1+\varphi_kf_k}{1+\psi_kg_k}\\
   &=\left(\frac{1+\varphi_tf_t}{1+\psi_tg_t}\right)\frac{\psi_t}{\varphi_t}\\
   &=\frac{1/\varphi_t+f_t}{1/\psi_t+g_t}\\
   &>\frac{\kappa /\psi_t+g_t+\varepsilon}{1/\psi_t+g_t}\\
   &\geq\frac{\kappa /\psi_t+M+\varepsilon}{1/\psi_t+M},
\end{align*}
where the latter inequality holds because  
$$\frac{d}{dz}\frac{\kappa /\psi_t+z+\varepsilon}{1/\psi_t+z}<0.$$
Furthermore, 
from the definition of $\kappa$ it is seen that 
$$\kappa (y+M)-(\kappa y+M+\varepsilon) = M(\kappa-1)-\varepsilon<0,$$
whence $$\frac{d}{dy}\frac{\kappa y+M+\varepsilon}{ y+M}<0.$$
 Thus, from the previous inequality we deduce that
\begin{align*}
  \prod_{k=t+1-r}^{t}\frac{1+\varphi_kf_k}{1+\psi_kg_k} &> \frac{\kappa/m+M+\varepsilon}{1/m+M}\geq (1+\eta)^{3r}.
\end{align*}

On the other hand, if   $\kappa\varphi_t\geq  \psi_t$, 
then using the fact that
$$(g_t+\varepsilon/2)\varphi_t \geq g_t\left(1+\frac{\varepsilon}{2M}\right)\varphi_t=g_t\kappa\varphi_t$$
combined with the inequality $\varphi_t\geq m$, 
it is verified that
\begin{align*}
  \frac{1+\varphi_tf_t}{1+\psi_tg_t}&> \frac{1+\varphi_t(g_t+\varepsilon)}{1+\psi_tg_t}\\
  &=\frac{1+\varphi_t(g_t+\varepsilon/2)}{1+\psi_tg_t}+\frac{\varphi_t\varepsilon/2}{1+\psi_tg_t}\\
  &\geq 1+\frac{m \varepsilon/2}{1+M}\\
  &\geq (1+\eta)^3. 
\end{align*}

\textit{Step 2.} Fix $t_2$ and $t_1$ as in the statement of the lemma. We claim that there exists a finite decreasing sequence $(h_n)_{0\leq n\leq N}\subset\R$ with the properties 
\begin{itemize} 
  \item  $h_0=t_2$;
  \item  $\kappa\varphi_{h_n}\leq  \psi_{h_n}$ for all $1\leq n\leq N-1$;
  \item  $\kappa\varphi_t\geq  \psi_t$ for all $t\in I=\bigcup_{n=0}^{N-1}[h_{n+1},h_n-\tau].$
\end{itemize} 
Indeed, fix   $h_0=t_2$ and while $t_1\leq h_{n } $ define
\begin{equation*}
  h_{n+1} =\left\{\begin{aligned}
    &h_{n }-r& \text{if } \kappa \varphi_{h_{n }-r}\leq  \psi_{h_{n }-r},\\
   &\inf\left\{\begin{array}{l}
   h\geq   t_1 -r: \kappa \varphi_t\geq \psi_t\\
   \text{for all }t\in[h,h_n-r]
   \end{array}\right\} &\text{otherwise.}
  \end{aligned}\right.
\end{equation*}
Observe that the sequence ends when   $t_1> h_{N } \geq t_1-r$ and  the case  $N=1$  is possible. From the previous step, it is obtained that  
\begin{align*}
  \prod_{k=h_n+1-r}^{h_n}\frac{1+\varphi_kf_k}{1+\psi_kg_k} &>   (1+\eta)^{3r} ,\\
  \frac{1+\varphi_tf_t}{1+\psi_tg_t}& > (1+\eta)^3  
\end{align*}
for all  $1\leq n\leq N-1$ and $t\in I$, which implies that 
$$\prod_{k=h_N+1}^{h_0-r}\frac{1+\varphi_kf_k}{1+\psi_kg_k}  >   (1+\eta)^{3(h_0-h_N-r)}.$$

\emph{Step 3.} To conclude, observe that
\begin{align*}
  \prod_{k=t_1+1-r}^{t_2-r}\frac{1+\varphi_kf_k}{1+\psi_kg_k}&\geq \left(\prod_{k=t_1+1-r}^{h_N}\frac{1}{1+M}\right) \prod_{k=h_N+1 }^{h_0-r}\frac{1+\varphi_kf_k}{1+\psi_kg_k} \\
  &> \frac{ (1+\eta)^{3(t_2-t_1+1-r)}}{\left(1+M \right)^{r-1}} \\ 
  &=  (1+\eta)^{2(t_2-t_1)}   \frac{ (1+\eta)^{t_2-t_1-3(r-1)}}{\left(1+M \right)^{r-1}}\\
  &\geq (1+\eta)^{2(t_2-t_1)} 
\end{align*}
since $t_2-t_1\geq T$, and the result is proved.
\end{proof}

The above lemmata  shall be employed  to prove that the first statement 
 in Theorem \ref{teoremapersistencia} implies the fourth one, 
 while the following one can be seen as a comparison lemma and 
 will be used to prove that statement \ref{d} implies statement \ref{c}.  

\begin{lem}\label{nuevolemavuelta}
Assume that $c_t$ and $x_t$ are positive for all $t\geq t_0-r$ and satisfy 
$$\left\{\begin{aligned}
c_{t+1}&=(1-E)(c_t+f_{t-r}c_{t-r}),\\
x_{t+1}&=(1-E)(x_t+g_{t-r}x_{t-r})
\end{aligned}\right.$$
for all $t\geq t_0$. Furthermore, assume there exists  $\varepsilon\in(0,\inf f)$ such that $|f_t-g_t|<\varepsilon$ if $t\geq t_0-r$ where $\inf f$ denotes the infimum of $f_t$. Therefore we obtain
\begin{equation}\label{lemavuelta}
  \frac{c_t}{x_t}\leq \left(\frac{\inf f}{\inf f-\varepsilon}\right)^{(t-t_0)/(r+1)} \max_{t_0\leq s\leq t_0+r}\left\{\frac{c_{s}}{x_{s}}\right\}.
\end{equation} 
\end{lem}

\begin{proof}
We shall proceed by induction. 
Notice that the result holds for all  $t_0\leq t\leq t_0+r$, and assume 
it is true for $t$ and $t-r$. There are two cases to study. On the one hand, if 
$$f_{t-r}\frac{c_{t-r}}{c_t}\geq g_{t-r}\frac{x_{t-r}}{x_t},$$
then
\begin{align*}
  \frac{c_{t+1}}{x_{t+1}}&=\frac{c_t+f_{t-r}c_{t-r}}{x_{t}+g_{t-r}x_{t-r}}\\
  &=\frac{c_t }{x_t }\frac{1+f_{t-r}c_{t-r}/c_{t }}{1+g_{t-r }x_{t-r }/x_{t }}\\
  &\leq \frac{c_t }{x_t } \frac{ f_{t-r}c_{t-r}/c_{t }}{ g_{t-r }x_{t-r }/x_{t }}\\
  &=\frac{f_{t-r}}{g_{t-r}}\frac{c_{t-r}}{x_{t-r}}\\
  &\leq \frac{f_{t-r}}{f_{t-r}-\varepsilon}\frac{c_{t-r}}{x_{t-r}}.
\end{align*}
Because
$$\frac{f_{t-r}}{f_{t-r}-\varepsilon}\leq \frac{\inf f}{\inf f-\varepsilon}$$
and using induction, it is verified that 
\begin{align*}
  \frac{c_{t+1}}{x_{t+1}}&\leq \frac{\inf f}{\inf f-\varepsilon} \left(\frac{\inf f}{\inf f-\varepsilon}\right)^{(t-r-t_0)/(r+1)} \max_{t_0\leq s\leq t_0+r}\left\{\frac{c_{s}}{x_{s}}\right\}\\
  &=\left(\frac{\inf f}{\inf f-\varepsilon}\right)^{(t+1-t_0)/(r+1)} \max_{t_0\leq s\leq t_0+r}\left\{\frac{c_{s}}{x_{s}}\right\}.
\end{align*}
On the other hand, if 
$$f_{t-r}\frac{c_{t-r}}{c_t}< g_{t-r}\frac{x_{t-r}}{x_t},$$
then it is seen that
\begin{align*}
  \frac{c_{t+1}}{x_{t+1}}&=\frac{c_t+f_{t-r}c_{t-r}}{x_{t}+g_{t-r}x_{t-r}}\\
  &=\frac{f_{t-r}c_{t-r} }{g_{t-r}x_{t-r} }\,\frac{c_t/(f_{t-r}c_{t-r})+1}{x_t/(g_{t-r }x_{t-r })+1}\\
  &\leq  \frac{c_{t}}{x_{t}}\\ 
  &<\left(\frac{\inf f}{\inf f-\varepsilon}\right)^{(t+1-t_0)/(r+1)} \max_{t_0\leq s\leq t_0+r}\left\{\frac{c_{s}}{x_{s}}\right\}.
\end{align*}
and so concludes the proof. 
\end{proof}


\begin{lem}\label{lema-atractividad}
The identity 
$$x_t=(1-E)^r(x_{t-r}+y_{t-r}) $$
holds for all $t\geq r$.
\end{lem}

\begin{proof}
We claim that 
$$x_t=(1-E)^{n+1}x_{t-1-n}+\sum_{k=0}^n (1-E)^{r+1+k}x_{t-1-r-k}p(s_{t-1-r-k}) $$
for all  $n\in[ 0,t-1]$. 
By induction, 
\begin{align*}
  x_t&=(1-E)^{n+1}x_{t-1-n}+\sum_{k=0}^n x_{t-1-r-k}p(s_{t-1-r-k})(1-E)^{r+1+k}\\
  &=(1-E)^{n+1}((1-E)x_{t-2-n}+x_{t-2-n-r}p(s_{t-2-n-r})(1-E)^{r+1})\\
  &\quad +\sum_{k=0}^nx_{t-1-r-k}p(s_{t-1-r-k}) (1-E)^{r+1+k}\\
  &=(1-E)^{n+2}x_{t-1-(n+1)}+\sum_{k=0}^{n+1} x_{t-1-r-k}p(s_{t-1-r-k})(1-E)^{r+1+k} .
\end{align*}
Moreover,  the case $n=0$ is direct and, when 
$n=r-1$, it is verified that  
\begin{align*}
  x_t&=(1-E)^{r}x_{t-r}+\sum_{k=0}^{r-1} (1-E)^{r+1+k}x_{t-1-r-k}p(s_{t-1-r-k})\\
  &=(1-E)^{r}\left(x_{t-r}+y_{t-r}\right).
\end{align*}
\end{proof}

\begin{lem}\label{lemaextincion}
Let $t_0\geq 0$ and consider 
\begin{align*}
  \varphi_{t+1}&=\prod_{k=t+1-r}^t (1+\varphi_kf_k)^{-1},\\
  \psi_{t+1}&=\prod_{k=t+1-r}^t (1+\psi_kg_k)^{-1}
\end{align*}
for all $t\geq t_0$, where $f_t\geq g_t$ for all $t\ge t_0$. Then
$$(1+M)^{r-1}\prod^{t_2-r}_{k=t_1+1-r}(1+\varphi_kf_k)\geq\prod^{t_2-r}_{k=t_1+1-r}(1+\psi_kg_k)$$
for all $t_2> t_1\geq t_0.$ 
\end{lem}

\begin{proof}
We shall split the proof in two steps.

\textit{Step 1.} Firstly, if $\varphi_{t+1}\leq \psi_{t+1}$ for all $t\geq t_0$, then
$$1\leq \frac{\psi_{t+1}}{\varphi_{t+1}}=\prod^{t }_{t +1-r}\frac{1+\varphi_kf_k}{1+\psi_kg_k}.$$
furthermore, if  $\varphi_{k}\geq \psi_{k}$ for all $k\in [h_1,h_2]$, then
$$\prod_{k=h_1 }^{h_2}\frac{1+\varphi_kf_k}{1+\psi_kg_k}\geq 1.$$

\textit{Step 2.} We shall repeat 
the idea used in Lemma \ref{lemaida}, now  setting $\kappa=1$, and  $t_2> t_1\geq t_0$. 
We claim that there is a finite and decreasing sequence $(h_n)_{0\leq n\leq N}\subset\R$ satisfying 
\begin{itemize} 
  \item $h_0=t_2$;
  \item  $ \varphi_{h_n}\leq  \psi_{h_n}$ for all $1\leq n\leq N-1$;
  \item  $ \varphi_t\geq  \psi_t$ if
  $$t\in I=\bigcup_{n=0}^{N-1}[h_{n+1},h_n-\tau].$$
\end{itemize} 
Next, set $h_0=t_2$ and, while $t_1\leq h_{n } $, set
\begin{equation*}
  h_{n+1} =\left\{\begin{aligned}
    &h_{n }-\tau& \text{if }   \varphi_{h_{n }-r}\leq  \psi_{h_{n }-r},\\
   &\inf\left\{\begin{array}{l}
   j \geq   t_1 -r:   \varphi_t\geq \psi_t\\
   \text{for all }t\in[j,h_n-r]
   \end{array}\right\} &\text{otherwise.}
  \end{aligned}\right.
\end{equation*}
The sequence ends when $t_1> h_{N } \geq t_1-r$ and  
$N=0$ is a possible outcome. From the previous step it is deduced that 
\begin{align*}
  \prod_{k=h_n+1-r}^{h_n}\frac{1+\varphi_kf_k}{1+\psi_kg_k} &\geq 1,\\
  \frac{1+\varphi_tf_t}{1+\psi_tg_t}&\geq 1 
\end{align*}
for all $1\leq n\leq N-1$ and $t\in I$ which, in turn, implies
\begin{align*}
  \prod_{k=t_1+1-r}^{t_2-r}\frac{1+\varphi_kf_k}{1+\psi_kg_k}&\geq \left(\prod_{k=t_1+1-r}^{h_N}\frac{1}{1+M}\right) \prod_{k=h_N+1 }^{h_0-r}\frac{1+\varphi_kf_k}{1+\psi_kg_k}\geq \frac{1}{(1+M)^{r-1}}
\end{align*} 
and the  proof is thus complete.
\end{proof}

\section{Proofs of the main results}
\label{proofs}
\subsection{The general case}

In order to give a complete proof of Theorem \ref{teoremapersistencia}, let us 
firstly observe 
that the following chain of implications holds trivially: 

$$
 \ref{c} \Longrightarrow 
\ref{b} 
\Longrightarrow \ref{a} \Longrightarrow  \ref{z}.
$$
Thus, it suffices to verify that \ref{z} implies \ref{d} and that, in turn, the latter statement implies \ref{c}.
To this end, let us  fix three constants $\underline S$, $\overline S$ and $M$ such that  $0<\underline S\leq s_t^{0}\leq \overline S$ for all $t$ and $M:=\max\{p(2\overline S)\}$, which constitutes an upper bound for the functions $f_t:=p(z_t)$ and $g_t:=p(s_t)$ for  $t\geq t_g$.  

\medskip 

\begin{proofpers}
In the first place, we shall demonstrate that   statement \ref{z} implies statement \ref{d}. The proof shall be based on  \cite[Lemma 3]{ellermeyer2001persistence}. 
{To this end, fix $(s_t,x_t)$  such that there is $\delta>0$ that satisfies $x_t\geq \delta$ for all $t\geq 0$.}
By Lemma \ref{lemaclave}, there exists $t_0$ such that $t\geq t_0$ implies 
$$z_t>s_t+x_t+y_t-\delta/2\geq s_t+\delta/2.$$
Recall that $s^0$ is bounded from above, therefore $z_t$ and $s_t$ are also bounded from above and there is a constant  $K$ such that  
$$\frac{p(z_t)-p(s_t)}{z_t-s_t}=p'(\xi)\geq K$$ 
and
$$p(z_t)>p(s_t)+K\delta/2.$$ 
By Lemma \ref{lemaida}, there are positive constants $\eta$  and $T$ such that  
$$ \prod_{k=t_1+1}^{t_2}\frac{1+\varphi_{k-r}p(z_{k-r})}{1+\psi_{k-r}p(s_{k-r})}>(1+\eta)^{2(t_2-t_1)}$$
for all $t_1>T$, $t_2-t_1>T$.

Furthermore, since $x_t$ is upperly bounded (this is a consequence of Lemma   \ref{lemaclave} and the fact that  $z_t$ is upper bounded) we can find $T$ large enough such that  
$$\frac{\delta}{x_{t }}>(1+\eta)^{-T} $$
for all $t\geq 0$. Finally, from the expression \eqref{expresionxt} it is seen that
\begin{align*}
  \prod_{k=t_1+1}^{t_2}(1+\varphi_{k-r}p(z_{k-r}))&>(1+\eta)^{2(t_2-t_1 )}\prod_{k=t_1+1}^{t_2}(1+\psi_{k-r}p(s_{k-r}))  \\
  &=(1+\eta)^{2(t_2-t_1 )}\frac{x_{t_2+1}}{x_{t_1+1}(1-E)^{t_2-t_1}}\\
  &\geq \left(\frac{1+\eta}{1-E}\right)^{ t_2-t_1  } \frac{\delta}{x_{t_1+1}}(1+\eta)^T\\
  &>\left(\frac{1+\eta}{1-E}\right)^{ t_2-t_1  },
\end{align*}
so the statement \ref{d} is proven.

Next, we shall show that statement \ref{d} implies statement \ref{c}. 
With this in mind, fix  positive constants $\eta$ and $T$  such that  
\begin{equation*}
(1-E)^{t_2-t_1}\prod_{k=t_1+1}^{t_2}(1+p(z_{k-r})\varphi_{k-r} )=\frac{c_{t_2+r+1}}{c_{t_1+r+1}}>(1+\eta)^{t_2-t_1} 
\end{equation*}
for all $t_1>T$ and $t_2-t_1>T$.  
The proof shall be  split  into two steps.

\textit{Step 1.}
Let us firstly prove the existence of a positive constant $C$ such that for  any particular trajectory $(s_t,x_t)$ and some sufficiently large value of $t$ it is verified that $x_t>C$. 
 With this in mind, fix $\varepsilon\in(0,p(\underline S))$ such that 
\begin{equation*}\label{condicionepsilon}
  \frac{p(\underline S)}{p(\underline S)-\varepsilon}=(1+\eta)^{(r+1)/2}
\end{equation*}
and define $L=\max_{s\in [0,2\overline S]}p'(s).$ Furthermore, let $R$ and $\alpha$ be any positive constants and fix $t_0\geq T+r+1$ such that
\begin{equation}\label{cotaparatodo}
  (\overline S+2R+Rp(R)r)(1-E)^{t_0-r}\leq \min\left\{\frac{\varepsilon}{2L},\overline S\right\} 
\end{equation}
and an initial time value $\tilde T:= \tilde T(R,\alpha) >t_0+T+r$ such that 
\begin{align*}
  \left(1+\eta\right)^{(\tilde T-t_0)/2-r}&\geq\frac{ \varepsilon}{2L(1+Mr)\alpha (1-E)^{t_0+r}}.
\end{align*}

We are now in position of working with a particular solution of system \eqref{modeloprincipal}. Therefore, consider   $(s_t,x_t)$ such that $\|(s^{in},x^{in})\|\leq R$ and $x_0\geq \alpha$. We claim that there exists $t^*\in [t_0-r,\tilde T]$ such that 
$$x_{t^*}\geq \frac{\varepsilon}{2L(1+Mr)}.$$
By contradiction, suppose that  
\begin{equation}\label{absurdo}
  x_t<\frac{\varepsilon}{2L(1+Mr)}
\end{equation} 
for all $t\in[t_0-r,\tilde T]$ and observe that 
\begin{equation}\label{cotaparaY}
  y_t=\sum_{k=1}^{r-1}x_{t-k}p(s_{t-k})(1-E)^k\leq \max_{k=1,\dots,r}\{x_{t-k}\}p\left(\max_{k=1,\dots,r}\{s_{t-k}\}\right)r  .
\end{equation}
In particular $y_0\leq Rp(R)r$, so we obtain
\begin{align*}
  |z_t-s_t-x_t-y_t|&=|z_0-s_0-x_0-y_0|(1-E)^t\\
  &\leq (\overline S+2R+Rp(R)r)(1-E)^{t}.
\end{align*}
Combined with \eqref{cotaparatodo}, this implies 
\begin{equation}\label{cota2s}
  s_t\leq |z_t-s_t-x_t-y_t|+z_t\leq 2\overline S 
\end{equation}
if $t\geq t_0-r$. Using again \eqref{cotaparatodo} and  \eqref{absurdo}, it follows  that 
\begin{align*}
  |z_t-s_t|&\leq |z_t-s_t-x_t-y_t|+x_t+y_t\\
  &<\frac{\varepsilon}{2L}+\frac{\varepsilon}{2L(1+Mr)}+\max_{k=1,\dots,r}\{x_{t-k}\}p\left(\max_{k=1,\dots,r}\{s_{t-k}\}\right)r\\
  &\leq \frac{\varepsilon}{2L}+\frac{\varepsilon}{2L(1+Mr)}+\frac{\varepsilon}{2L(1+Mr)}p(2\overline S)r\\
  &=\varepsilon/L
\end{align*} 
if $t\geq t_0 $, which implies that 
\begin{align*}
  |p(z_t)-p(s_t)|&\leq L |z_t-s_t|<\varepsilon.  
\end{align*}

Finally, fix $s^*\in [t_0,t_0+r]$ such that 
$$\max_{t_0\leq s\leq t_0+r}\left\{\frac{c_{s}}{x_{s}}\right\}= \left\{\frac{c_{s^*}}{x_{s^*}}\right\}$$
and observe that    $z_t\geq \underline S$ implies that $p(z_t)\geq f(\underline S)$. 
Thus,  Lemma   \ref{nuevolemavuelta} yields a contradiction, since 
\begin{align*}
  x_{\tilde T}&=x_0\frac{x_{s^*}}{x_0} \frac{x_{\tilde T}}{x_{s^*}}\\
  &\geq \alpha \left(\prod_{k=0}^{s^*-1}(1-E)\right) \frac{x_{\tilde T}}{x_{s^*}} \\
  &\geq \alpha (1-E)^{t_0+r} \left(\frac{f(\underline S)-\varepsilon}{f(\underline S)}\right)^{(\tilde T-t_0)/(r+1)}\frac{c_{\tilde T}}{c_{s^*}} \\
  &\geq \alpha (1-E)^{t_0+r} \left(1+\eta\right)^{-(\tilde T-t_0)/2}(1+\eta)^{\tilde T-s^*} \\
  &\geq \alpha (1-E)^{t_0+r} \left(1+\eta\right)^{-(\tilde T-t_0)/2}(1+\eta)^{\tilde T-t_0-r} \\ 
  &\geq \frac{\varepsilon}{2L(1+Mr)}.
\end{align*} 

\textit{Step 2.} In order to complete the proof, we need to find a positive lower bound 
for arbitrary initial values and sufficiently large times. 
Let 
$$\hat T:=\tilde T\left(\sqrt{8} \overline S ,\frac{\varepsilon}{2L(1+Mr)}\right),$$
and define 
$$\delta:=\frac{\varepsilon}{2L(1+Mr)(1-E)^{\hat T}},$$
which is independent of the initial conditions. 

From the previous step, we have that 
\begin{itemize}
  \item $s_t\leq 2\overline S$ if $t\geq t_0-r$,
  \item $x_t\leq 2\overline S$ if $t\geq t_0-r$, which is verified exactly as  \eqref{cota2s},
  \item there exists $t^*_0\geq t_0-r$ such that $x_{t^*_0}\geq \varepsilon/(2L(1+Mr))$.
\end{itemize}
Then, the particular solution $(s_{t-t^*_0},x_{t-t^*_0})$ satisfies   
$$|(s_{t-t^*_0},x_{t-t^*_0})|\leq \sqrt{8}\overline S$$
for all $t\in [t^*_0-r,t^*_0]$ and  $x_{t^*_0}\geq \varepsilon/(2L(1+Mr))$. Furthermore, there exists $t_1^*\in (t^*_0,\tilde T]$ such that   
$x_{t^*_1}\geq \varepsilon/(2L(1+Mr)).$ 
Inductively, we obtain $(t^*_n)_{n\in\N_0}$ such that $t^*_{n+1}-t^*_{n}\leq \hat T$ and  $x_{t^*_n}\geq  \varepsilon/(2L(1+Mr)).$ 
Finally, for $t\geq \tilde T\geq t^*_0-r$ it is deduced that
$$x_t\geq (1-E)^{-\hat T}\frac{\varepsilon}{2L(1+Mr)}=\delta,$$
as desired.
\end{proofpers}

{\begin{rem} It is noticed, in the previous proof, that a value   $\tilde T$ was established in order to satisfy 
$$\alpha (1-E)^{t_0+\tilde T}(1+\eta)^{I^{in}/2}\geq\frac{\varepsilon}{2L(1+Mr)}.$$
It is worthy to compare this with the analogous requirement  in \cite{ellermeyer2001persistence} for the continuous case, namely 
$$\alpha e^{-E(t_0+\tilde T)} e^{\eta I^{in}/2}\geq\frac{\varepsilon}{2L(1+Mr)}.$$
 \end{rem}}

\begin{proofattr}
Let   $(s_r,x_t)$ and   $(\hat s_t,\hat x_t)$ be  solutions. We split the proof into four steps.

\textit{Step 1.} We shall construct a function  $w_t$ that  lies above the function
$s_t-\hat s_t$ and such that $w_t\to 0$ as $t\to +\infty$. To this end, fix a positive constant 
$$\gamma=\min_{s\in [0,2\overline S]}\{\delta p'(s),1\}$$
where $\delta>0$ is given by Theorem \ref{teoremapersistencia}. Since $s_t$ is upper bounded (because $s^0_r$ is upper bounded) we can define the function  
\begin{align*}
  J(\varepsilon)&=1- (1-E)^\varepsilon+(1-E)^{-\varepsilon r}\max_{t\geq 0}\{p(s_t)\}[1-(1-E)^{\varepsilon r+\varepsilon}] -\frac{ \gamma}{2}
\end{align*} 
that is continuous and satisfies $J(0)<0$. Therefore, we can fix $\varepsilon\in (0,1/2]$ such that $J(\varepsilon)\leq 0$. 

By Lemma \ref{lemaclave} and Lemma \ref{lema-atractividad} we have that
\begin{align*}
    \hat x_t-x_t&=(1-E)^r(\hat x_{t-r}+\hat y_{t-r}+\hat s_{t-r}-z_{t-r}+z_{t-r}-x_{t-r}-y_{t-r}-s_{t-r}\\
    &\quad +s_{t-r}-\hat s_{t-r})\\
    &=(1-E)^tC+(1-E)^r( s_{t-r}-\hat s_{t-r})
\end{align*}
where $C=(1-E)^{-r}(\hat x_0+\hat y_0+\hat s_0-x_0-y_0-s_0).$ Therefore we can fix $t_0$ such that
\begin{equation}\label{desigualdadt0}
    (1-E)^{t/2} C \max_{t\geq 0}\{p(s_t)\}<\delta \gamma/2
\end{equation}
if $t\geq t_0$. 

\emph{Step 2.} Now define  
$$w_t=K(1-E)^{\varepsilon t}x_{t+r} $$
with $K\geq 1$ large enough such that $s_{t }\leq w_{t }$ for all $t\leq t_0$. By Lemma \ref{lema-atractividad} we get
\begin{align*}
    w_t&=K(1-E)^{\varepsilon(t-r+r)}\frac{x_{t+r}}{x_t}x_t\\
    &=(1-E)^{\varepsilon r}\frac{ x_{t+r}}{x_t}w_{t-r}\\
    &=(1-E)^{\varepsilon r}\frac{(1-E)^r (x_{t}+y_t)}{x_t}w_{t-r}\\
    &\geq (1-E)^{\varepsilon r+r}w_{t-r}.
\end{align*}
And then
\begin{equation}\label{desigualdadwt}
    \begin{aligned}
   & w_t [ 1-(1-E)^\varepsilon]  +(1-E)^{r }w_{t-r} p(s_t) [ 1-(1-E)^{\varepsilon r+\varepsilon}] \\
    &\leq w_t\Big( 1-(1-E)^\varepsilon +(1-E)^{-\varepsilon r }  p(s_t) [ 1-(1-E)^{\varepsilon r+\varepsilon}] \Big)\\
    &\leq w_t (J(\varepsilon)+\gamma/2)\\
    &\leq w_t\gamma/2
    \end{aligned}
\end{equation}
Furthermore,
\begin{align*}
  w_{t+1}&=K(1-E)^{\varepsilon t+\varepsilon}x_{t+1+r} \\
  &=K(1-E)^{\varepsilon t+\varepsilon}[(1-E)x_{t+r}+x_tp(s_t)(1-E)^{r+1}]\\
  &=(1-E)^{\varepsilon+1}w_t+(1-E)^{ \varepsilon r+\varepsilon+r+1}w_{t-r}p(s_t).
\end{align*}

\textit{Step 3.} We claim that   $s_t-\hat s_t\leq w_t.$ For this we proceed by induction in the time variable. Note that for all $t\leq t_0$  
$$s_t-\hat s_t\leq s_t \leq w_t.$$
Now assume   $s_s-\hat s_s\leq w_s$ for all $s\geq t$, then
\begin{align*}
    s_{t+1}-\hat s_{t+1}&=(1-E)[s_t-\hat s_t-x_tp(s_t)+\hat x_tp(\hat s_t)]\\
    &=(1-E)[s_t-\hat s_t-x_tp(s_t)+\hat x_t p(s_t)-\hat x_t p(s_t)+\hat x_tp(\hat s_t)]\\
    &=(1-E)[s_t-\hat s_t+(\hat x_t-x_t)p(s_t)+\hat x_t( p(\hat s_t)- p(s_t))]\\
    &=(1-E)[s_t-\hat s_t+(1-E)^tCp(s_t)\\
    &\quad +(1-E)^r( s_{t-r}-\hat s_{t-r})p(s_t)+\hat x_t p'(\xi_t)(\hat s_t- s_t)]\\
    &=(1-E)(s_t-\hat s_t) [1-\hat x_tp'(\xi_t)]\\
    &\quad +(1-E)^{r+1}( s_{t-r}-\hat s_{t-r})p(s_t)+(1-E)^{t+1}Cp(s_t)
\end{align*}
where we used  Lagrange's mean value theorem with $\xi_t$ between $\hat s_{t-r}$ and $s_{t-r}$ and, in particular,
 $\xi_t\in [0,2\overline S]$. Therefore, by induction and Theorem \ref{teoremapersistencia}:
\begin{align*}
  s_{t+1}-\hat s_{t+1}&\leq (1-E)w_t [1-\gamma]+(1-E)^{r+1}w_{t-r} p(s_t) +(1-E)^{t+1}Cp(s_t)\\
  &=(1-E)w_t [(1-E)^\varepsilon+1-(1-E)^\varepsilon] \\
  &\quad +(1-E)^{r+1}w_{t-r} p(s_t) [(1-E)^{\varepsilon r+\varepsilon}+1-(1-E)^{\varepsilon r+\varepsilon}] \\
  &\quad +(1-E)^{t+1}Cp(s_t)-(1-E)w_t  \gamma\\
  &=w_{t+1}+(1-E)w_t [ 1-(1-E)^\varepsilon] \\
  &\quad +(1-E)^{r+1}w_{t-r} p(s_t) [ 1-(1-E)^{\varepsilon r+\varepsilon}] \\
  &\quad +(1-E)^{t+1}Cp(s_t)-(1-E)w_t  \gamma .
\end{align*}
Now we use the inequalities \eqref{desigualdadt0} and
\eqref{desigualdadwt} to deduce:
\begin{align*}
    s_{t+1}-\hat s_{t+1}&\leq w_{t+1}+(1-E)^{t+1}Cp(s_t)-(1-E)w_t  \gamma/2\\
    &\leq w_{t+1}+(1-E)[(1-E)^tC\max_{t\geq 0}\{p(s_t)\}-K(1-E)^{\varepsilon t}x_{t+r}\gamma/2]\\
    &\leq w_{t+1}+(1-E)[(1-E)^tC\max_{t\geq 0}\{p(s_t)\}- (1-E)^{ t/2}\delta\gamma/2]\\
    &= w_{t+1}+(1-E)^{1+t/2}[(1-E)^{t/2}C\max_{t\geq 0}\{p(s_t)\}-  \delta\gamma/2]\\
    &\leq w_{t+1},
\end{align*}
as desired.

\textit{Step 4.} Similarly, $\hat s_t-s_t$ is upper bounded and we get that $|s_t-\hat s_t|$ goes to zero when $t$ tends to infinity. By Lemma \ref{lemaclave} this implies that 
$$|x_t+y_t-(\hat x_t+\hat y_t)|$$
tends to zero and, again by Lemma \ref{lema-atractividad}, we obtain that 
$|x_t-\hat x_t|$
goes to zero as desired.
\end{proofattr}

\begin{proofext}
By hypothesis, there exist $\eta>0$ and $T>r$ such that   
$$ 
  (1-E)^{t_2-t_1}\prod_{k=t_1+1}^{t_2}(1 +\varphi_{k-r}p(z_{k-r}) )<\left( 1-\eta \right)^{ t_2-t_1  }
$$
for all $t_1>T$ and $t_2-t_1>T$. Now fix $(s,x)$ solution. If extinction occurs, then there is nothing to prove. Otherwise,  it happens that
$\displaystyle\limsup_{t\to\infty}x_t>0$ and, by Lemma \ref{lemaclave}, there exists $t_0$ such that 
$z_t\geq s_t$ for all $t\geq t_0$.  By  \eqref{expresionxt} and Lemma \ref{lemaextincion}, 
\begin{align*}
  x_{t_2+1}& =x_{ t_1+1}(1-E)^{t_2-t_1}\prod_{k=t_1+1-r}^{t_2-r}( 1 +\psi_{k } p(s_{k }))\\
  &\leq x_{ t_1+1}(1-E)^{t_2-t_1}(1+M)^{r-1}\prod_{k=t_1+1-r}^{t_2-r}( 1 +\varphi_{k } p(z_{k }))
  \\   &<x_{ t_1+1}(1+M)^{r-1}(1-\eta)^{t_2-t_1}
\end{align*}
which tends to zero as $t_2$ goes to infinity,  provided that $t_1$ is large enough. This contradicts the assumption that $\displaystyle\limsup_{t\to\infty}x_t>0$ and so completes the proof.
\end{proofext}

\subsection{Periodic case}

For the following proof we shall apply the Horn  fixed point theorem \cite[Theorem 6]{horn1970some}.  
\smallskip 


\begin{proofperiod}
According to the results for the general 
case, it only remains to prove that the 
persistence condition 
$\langle (1-E) (1+\varphi p(z))
\rangle>1$ implies the existence of 
at least one positive $\omega$-periodic solution. 
Indeed, observe that, due to Theorem \ref{atractividad}, such a solution attracts all the positive trajectories and, in particular, it is 
unique. Conversely, if there exists a positive 
$\omega$-periodic solution, then it is persistent 
and Theorem \ref{teoremapersistencia}
applies. 

Consider the Banach space
$$\mathcal{F}=\{ \phi:\{-r,-r+1,\dots,-1,0\}\to\R^2\}$$
equipped with the norm 
$$\|\varphi\|=\max_{-r\leq t\leq 0}\sqrt{\phi_1(t)^2+\phi_2(t)^2}. $$

\textit{Step 1.}  Let  $\delta>0$ be the persistence bound given by Theorem   \ref{teoremapersistencia} and set $R=3\overline S$, $\alpha= (1-E)^{-r}\delta/2$, $T^{in}=T^{in}(R,\alpha)$, and $R_0=z_0+6\overline S+3\overline Sp(3\overline S)r+\overline S$. We claim that  
$$\sqrt{s_t^2+x_t^2}\leq R_0$$
for all $t\geq 0$ if $\|(s^{in},x^{in})\|\leq 3\overline S$.
Indeed, recall the inequality   \eqref{cotaparaY} and notice that  
$$y_0\leq \|(s^{in},x^{in})\|p(\|(s^{in},x^{in})\|)r\leq 3\overline Sp(3\overline S)r.$$
By Lemma \ref{lemaclave}, it is deduced that
\begin{equation}\label{cotaSyx}
\begin{aligned}
  |(s_t,x_t)|&\leq s_t+x_t\\
  &\leq s_t+x_t+y_t\\
  &\leq |z_t-s_t-x_t-y_t|+z_t\\
  &\leq (1-E)^{-t}|z_0-s_0-x_0-y_0|+\overline S\\
  &\leq \max\{z_0,s_0+x_0+y_0\}+\overline S\\
  &\leq R_0.
\end{aligned}
\end{equation}

\textit{Step 2.} Define the sets
\begin{align*}
  s_0&=\{\phi\in \mathcal{F}:\|\phi\|\leq 2\overline S,\, \phi_2(0)\geq \delta\},\\
  s_1&=\{\phi\in \mathcal{F}:\|\phi\|< 3\overline S,\, \phi_2(0)> \delta/2\},\\
  s_2&=\{\phi\in \mathcal{F}:\|\phi\|\leq R_0,\, \phi_2(0)\geq (1-E)^{-T^{in}}\delta/2 \}.
\end{align*}

In order to apply Horn's Theorem, note that   $s_0\subset s_1\subset s_2\subset\mathcal{F}$ are convex bounded sets. Further, 
 $s_0$ and $s_2$ are closed and consequently compact, because $\mathcal F$ is finite dimensional.   
Moreover, observe that $s_1$ is open.  

Next, define the Poincar\'e Operator 
\begin{align*}
  P:s_2&\to \mathcal{F}\\
  \phi&\mapsto (t\mapsto (s_{t+\omega},x_{t+\omega})), \quad t=-r,\ldots, 0,
\end{align*}
where $(s_t,x_t)$ is the solution of system \eqref{modeloprincipal} with initial condition $\phi$. Then $P^k(\phi)_t$ is the solution for $k\omega-r\leq t\leq k\omega$ with the same initial condition $\phi$.

\textit{Step 3.} We claim that $P^k(s_1)\subset  s_2$ for all $k\geq 0$ and $P^k(s_1)\subset  s_0$ from some  $k\geq m$ with $m\in\N$. 
Regarding the first claim, observe that the initial condition $\phi$ is bounded from above  by   $3\overline S$, so  the solution is upperly bounded by $R_0$. Further, if $\phi_2(0)>\delta /2$, using  \eqref{expresionxt} it happens that
$$x_t\geq (1-E)^{-t}\delta/2>(1-E)^{-\tilde T}\delta/2$$
for all $t< \tilde T$ and  
$$x_t\geq \delta >(1-E)^{-\tilde T}\delta/2$$
for all $t\geq \tilde T$. This proves that $P^k(s_1)\subset s_2$ for all $k\geq 0$.

For the second claim, take $T\geq T^{in}$ such that 
$$R_0(1-E)^{-T}\leq \overline S.$$
Then, if the initial condition is  bounded from above by  $3\overline S$, similarly as in \eqref{cotaSyx}, we obtain
$$|(s_t,x_t)|\leq (1-E)^{-t}|z_0-s_0-x_0-y_0|+\overline S\leq 2\overline S$$
for $t\geq T$. And $x_t\geq \delta$ if $t\geq T$. Therefore, the claim is deduced  after setting  $m\geq (T+r)/\omega$.

Thus, all the assumptions of   Horn's Theorem are fulfilled, so we conclude that  
$P$ has at least one 
fixed point in $s_2$, which corresponds 
to a  positive $\omega$-periodic solution of the problem. 
\end{proofperiod}

\medskip

The following proof is based on   \cite[Theorem 1]{amster2020dynamics} and \cite[Theorem 2.2]{cartabia2025uniform}.
\smallskip

\begin{proofperext}
As before, it is observed that, due to  
Theorem \ref{extinciongeneral}, it only remains to analyze the case 
 $$\langle (1-E) (1+\varphi p(z))
\rangle = 1.$$
We claim that
\begin{equation}\label{contradiccionextincion}
    \prod_{k=t_0-r}^{t-r}\frac{1 +\psi_{k } p(s_{k })}{1+\varphi_kp(z_k)}
\end{equation} 
goes to zero when $t$ goes to infinity. By contradiction, suppose there it $\varepsilon>0$ such that, for all $t\geq t_0$ there exists $t_1>t$ such that
$$\prod_{k=t_0-r}^{t_1-r}\frac{1 +\psi_{k } p(s_{k })}{1+\varphi_kp(z_k)}>\varepsilon.$$
Note that by Lemma \ref{lemaextincion}
\begin{align*}
    \prod_{k=t_0-r}^{t-r}\frac{1 +\psi_{k } p(s_{k })}{1+\varphi_kp(z_k)}&=\left(\prod_{k=t_0-r}^{t_1-r}\frac{1 +\psi_{k } p(s_{k })}{1+\varphi_kp(z_k)}\right)\prod_{k=t+1-r}^{t_1-r}\frac{1+\varphi_kp(z_k)}{1 +\psi_{k } p(s_{k })}\\
    &>\varepsilon\frac{1+\varphi_kp(z_k)}{1 +\psi_{k } p(s_{k })}\\
    &\geq \varepsilon (1+M)^{1-r}
\end{align*}
Therefore,
 \begin{align*}
   x_{t+1}& =x_{ t_0 }(1-E)^{t+1-t_0}\prod_{k=t_0 -r}^{t-r}( 1 +\psi_{k } p(s_{k }))\\
  &=x_{ t_0 }\left(\prod_{l=t_0 -r}^{t-r}[1+\varphi_lp(z_l)](1 -E)\right)\prod_{k=t_0 -r}^{t-r}\frac{1 +\psi_{k } p(s_{k })}{1+\varphi_kp(z_k)} \\
  &> x_{ t_0 }\left(\min_{l\in [0,\omega]}[1+\varphi_lp(z_l)](1 -E)\right)^\omega \varepsilon (1+M)^{1-r}.
 \end{align*}
 This implies that 
 $\displaystyle\liminf_{t\to\infty}x_t>0$. By Theorem \ref{teoremapersistencia} System \eqref{modeloprincipal} is persistent, so a contradiction yields. Thus, \eqref{contradiccionextincion} is proved.

Finally, it is verified that
 \begin{align*}
   x_{t+1}& =x_{ t_0 }(1-E)^{t+1-t_0}\prod_{k=t_0 -r}^{t-r}( 1 +\psi_{k } p(s_{k }))\\
  &=x_{ t_0 }\left(\prod_{l=t_0 -r}^{t-r}[1+\varphi_lp(z_l)](1 -E)\right)\prod_{k=t_0 -r}^{t-r}\frac{1 +\psi_{k } p(s_{k })}{1+\varphi_kp(z_k)} \\
  &\leq  x_{ t_0 }\left(\max_{l\in [0,\omega]}[1+\varphi_lp(z_l)](1 -E)\right)^\omega\prod_{k=t_0 -r}^{t-r}\frac{1 +\psi_{k } p(s_{k })}{1+\varphi_kp(z_k)} 
 \end{align*}
 goes to zero when $t$ goes to infinity and the proof is complete.
\end{proofperext}

\section{Numerical simulations}\label{simulaciones}

In this section, we 
present some  some numerical simulations to illustrate our results. We have implemented them in GNU Octave \cite{octave}. Further, in the left image of Figure \ref{figura1} and in the two images of Figure \ref{figura2} we plot $s^0$ with a black dashed line, $s$ with a blue dotted line, and $x$ with a red solid line.

In the first simulation, in the left image of Figure \ref{figura1} we set the function $s^0$ to be initially constant and then linearly decreasing. In the  image on the right, we analyze condition \eqref{condicionpersistencia}: for each time value $t$ we plot the quantity
$$\prod_{k=\lfloor t/2\rfloor}^t\left(\left[ 1+\varphi_{k-r}p(z_{k-r})\right](1-E)\right)$$
with a solid blue  line. We remark that as long as $s^0$ is constant, the concentration of biomass persists and this function is larger than the threshold value $1$, represented by the dashed red  line. However, this function decreases when $s^0$ decreases and the concentration of biomass goes to zero.

\begin{figure}[htp]
\begin{center}
  \begin{minipage}{0.49\textwidth}
  \centering
	  \includegraphics[width=2in]{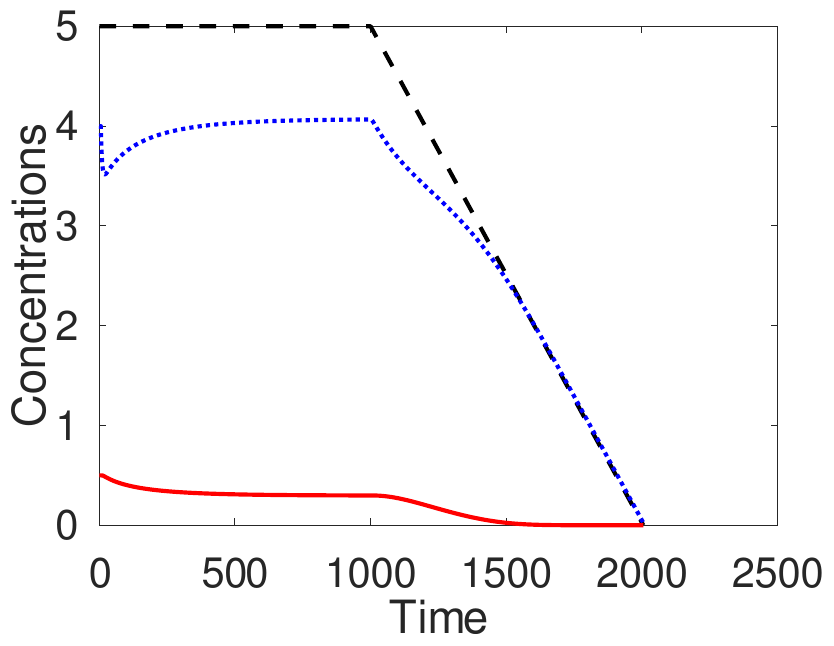}
	\end{minipage}
	\begin{minipage}{0.49\textwidth}
	\centering
	   \includegraphics[width=2in]{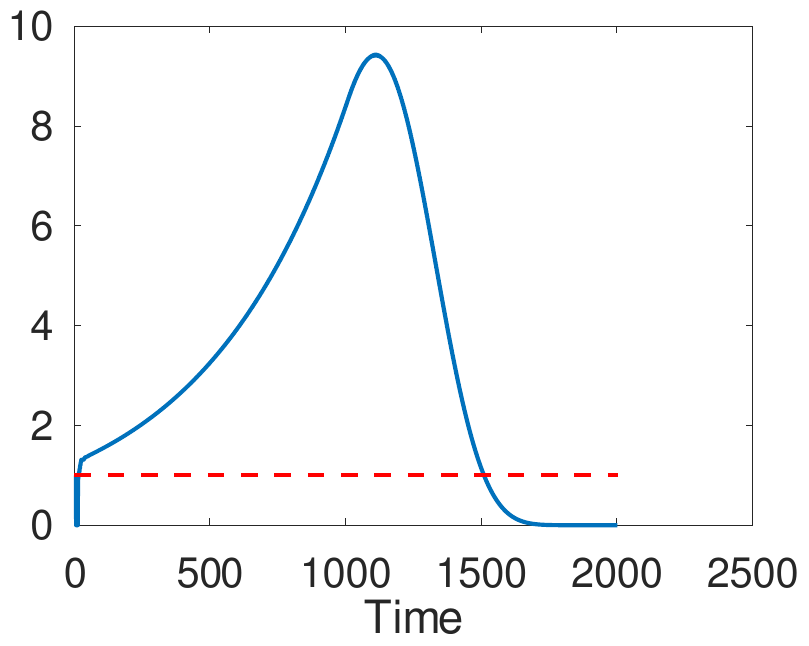}
	\end{minipage}	
	\caption{Numerical simulation of system \eqref{modeloprincipal} with $r=5$, $E=1/5.5$, $p(x)=x(x+1)^{-1}$, and constant initial conditions  $s_0\equiv 1/4$ and $x_0\equiv 1/2$.}\label{figura1}	
\end{center}
\end{figure}

In Figure \ref{figura2}  the periodic case is considered. For each time value $t$ we set
$$s^0_t=\sin(2\pi t/500)/4+a$$
with $a=0.6$ in the  image on the left and $a=0.3$ in the image on the right. Furthermore, the average defined in the last statement of Theorem \ref{existenciadesolucion} is computed as $1.0217$ and $0.9756$, respectively. This is consistent with the existence of an attractive periodic positive solution of $x$ in the first case and the extinction of the biomass in the second one.

\begin{figure}[htp]
\begin{center}
  \begin{minipage}{0.49\textwidth}
  \centering
	  \includegraphics[width=2in]{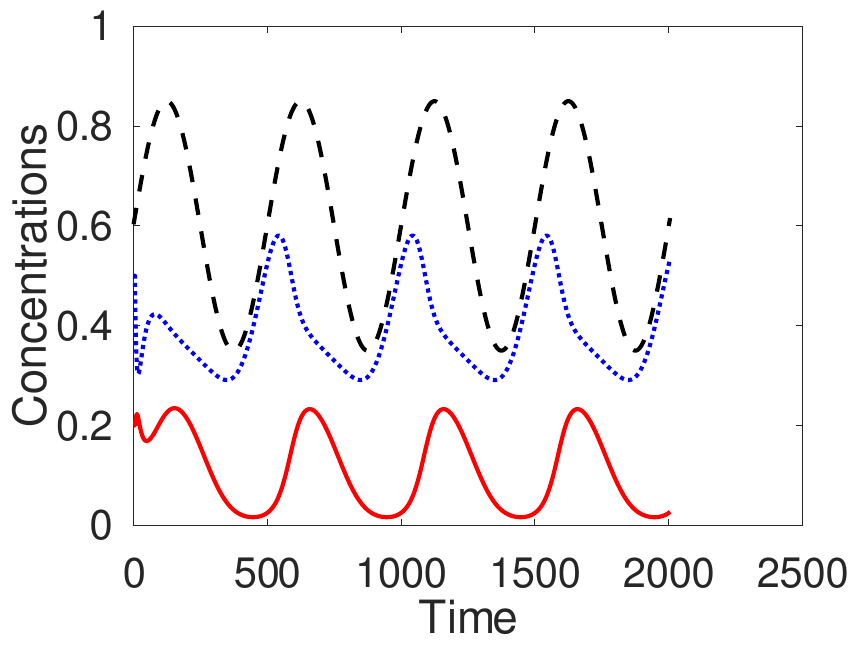}
	\end{minipage}
	\begin{minipage}{0.49\textwidth}
	\centering
	   \includegraphics[width=2in]{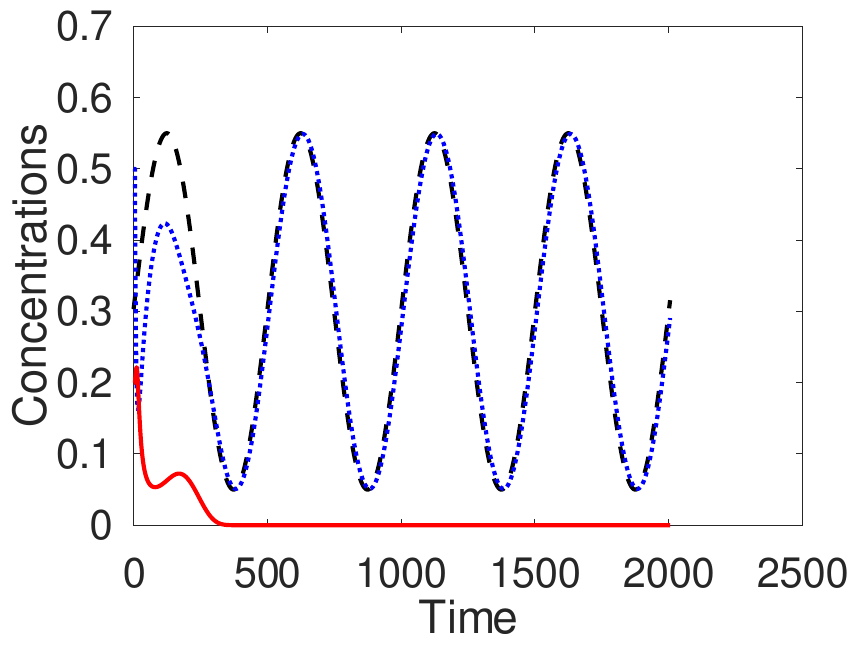}
	\end{minipage}	
	\caption{Numerical simulation of the periodic case with $r=5$, $E=1/8$, $p(x)=x(x+1)^{-1}$, and initial conditions  $s\equiv 1/2$ and $x\equiv 1/5$.
    }
    \label{figura2}	
\end{center}
\end{figure}

\section{Conclusions}

The dynamics of bioreactors remain a focal point in mathematical biology, both for their broad applications and the unresolved theoretical challenges they still present. 
The results in the present paper aim to build a bridge between  theoretical models  and  its computational aspects, offering tools to refine bioreactor designs and interpret numerical simulations. 

Specifically, a non-autonomous discrete delayed system modeling a one-species chemostat was studied, inspired by the continuous-time framework proposed in the well established paper \cite{ellermeyer1994competition}.
Conditions for the persistence or extinction of the solutions are established, 
characterized in terms of the lower and upper Bohl exponents of an associated scalar linear equation. To our knowledge, this approach constitutes a novelty  for discrete delayed chemostat systems and extends previous results in the literature in two directions. On the one hand, the connection with the Bohl exponents 
was not explored in the continuous case {and}, furthermore, the uniform persistence of the system was deduced from the existence of a single persistent solution. On the other hand, unlike the delayed system 
considered in \cite{AR24}, in the Ellermeyer model the results from the continuous case can be fully retrieved without imposing extra assumptions.

\section*{Acknowledgements}
The authors want to express their gratitude to the anonymous reviewers for the careful reading of the manuscript and the fruitful suggestions and comments. 
This work was partially supported by projects UBACyT 20020190100039BA and  PIP 11220200100175CO, CONICET.
 

\bibliographystyle{plain}

\end{document}